\documentclass[review]{elsarticle}

\usepackage{lineno,hyperref}
\modulolinenumbers[5]

\journal{Computational Statistics and Data Analysis}

\usepackage{booktabs}
\usepackage{float}
\usepackage{rotating}
\usepackage{epstopdf}
\usepackage{graphicx}
\usepackage{caption}
\usepackage{subcaption}
\usepackage{tabularx}
\usepackage{amsmath, bbm}
\usepackage{amssymb}
\usepackage{amsthm}
\usepackage{anysize}
\usepackage{mathrsfs}
\usepackage{multicol}
\usepackage{multirow}

\biboptions{authoryear}

\newcommand{\bbE}{\mbox{$\mathbb{E}$}}

\newcommand\raiseT[2]{%
\setbox0\hbox{$#1{#2}$}\raise\dp0\box0}

\newcommand{\bea}{\begin{eqnarray}}
\newcommand{\eea}{\end{eqnarray}}
\newcommand{\ba}{\begin{eqnarray*}}
\newcommand{\ea}{\end{eqnarray*}}
\newcommand{\be}{\begin{equation}}
\newcommand{\ee}{\end{equation}}
\newcommand{\bi}{\begin{itemize}}
\newcommand{\ei}{\end{itemize}}

\newcommand{\ind}{\mathbbm{1}}

\newtheorem{theorem}{Theorem}[section]

\newtheorem{lemma}{Lemma}[section]

\begin{document}

\begin{frontmatter}

\title{Consistency of the  MLE under a two-parameter gamma mixture model
with a structural shape parameter
}


\author[mymainaddress]{Mingxing He}

\author[mysecondaryaddress,mythirdaddress]{Jiahua Chen\corref{mycorrespondingauthor}}
\cortext[mycorrespondingauthor]{Corresponding author}
\ead{jhchen@stat.ubc.ca}

\address[mymainaddress]{Yunnan Key Laboratory of Statistical Modeling and Data Analysis,Yunnan University, Kunming {\rm 650091}, China}
\address[mysecondaryaddress]{Research Institute of Big Data, Yunnan University, Kunming  {\rm 650221}, China}
\address[mythirdaddress]{Department of Statistics, University of British Columbia, Vancouver  {\rm V7C 5K5}, Canada}

\begin{abstract}
The finite Gamma mixture model is often used to describe randomness
in income data, insurance data, and data from other applications.
The popular likelihood approach, however, does not work
for this model because the likelihood function
is unbounded, and the maximum likelihood estimator is therefore
not well defined.
There has been much research into ways to ensure the consistent estimation of
the mixing distribution, including placing an upper bound on
the shape parameter or adding a penalty to the log-likelihood function.
In this paper, we show that if the shape parameter in the
finite Gamma mixture model is structural, then
the maximum likelihood estimator of the mixing distribution is well defined
and strongly consistent.
We also present simulation results demonstrating
the consistency of the estimator.
We illustrate the application of the
model with a structural scale parameter to household income data.
The fitted mixture distribution leads to several
possible subpopulation structures
in terms of the level of disposable income.
\end{abstract}

\end{frontmatter}

\linenumbers

\section{Introduction}
Finite mixture distributions are widely used to model data collected
from a heterogeneous population: the population contains several subpopulations,
and each can be modeled by a distribution from a parametric distribution family.
Let $\{f(x; \theta): \theta \in \Theta\}$ be the density functions
of a parametric distribution family with parameter space $\Theta$.
A finite mixture distribution of order $m$ on this family has the density function
\be
\label{general.mixture}
f(y;  G) = \sum_{j=1}^m \alpha_j f(y; \theta_j)
\ee
for mixing proportions $\alpha_j \in [0, 1]$ such that $\sum_{i=1}^m \alpha_j = 1$,
and $m$ subpopulation parameters $\theta_j \in \Theta$.
We refer to $G$ as a mixing distribution that assigns probability $\alpha_j$ to
value $\theta_j$ in the subpopulation parameter space $\Theta$.
We refer to $f(y; \theta)$ as the subpopulation or component density function.
When we include all the distributions $G$ on $\Theta$, not merely those with a finite
number of support points, and we replace the summation in \eqref{general.mixture}
with integration, a general mixture model emerges.
This paper, however, will focus on a special class of finite mixture models.

There is a rich statistical literature on the theory and applications of mixture models.
More than a hundred years ago, \cite{pearson1894} used a finite Gaussian mixture
model of order two to model crab data suspected to consist of two species.
He used the method of moments because it is less computationally demanding.
Most recently, \cite{mclachlan2019finite} provided a thorough review of finite
mixture models.
\cite{KW1956} proved the general consistency as $n \to \infty$ of the maximum likelihood
estimator (MLE) of $G$ given a set of independent and identically distributed
(IID) samples of size $n$. Hereafter, we may refer to an IID sample
as a random sample.
\cite{dempster1977maximum} introduced the EM-algorithm,
an easy-to-implement numerical method to find the MLE for
finite mixture models. The convergence of this algorithm
was thoroughly discussed in \cite{Wu1983}.

The MLE consistency result given by \cite{KW1956} does not
apply to some important cases. Notably, the MLE of $G$ under
the finite normal mixture model is not well defined because its likelihood
function based on a random sample is unbounded.
To estimate $G$ consistently via the likelihood approach,
one may apply a regularizing penalty function
to the likelihood. The consistency of a penalized MLE was
rigorously established in \cite{Chen2008} and \cite{Chen2009}
for the univariate and multivariate cases respectively;
see \cite{Ciuperca2003} and \cite{Tanaka2009} for related developments.
An alternative approach is to place constraints on the range of the component
parameters $\Theta$. For instance, the MLE is consistent when
the ratio of any two-component variances is bounded by a prespecified constant:
see \cite{Hathaway1985}, \cite{Tanaka2005}, \cite{Tanaka2006}, and \cite{Chen2016}.
A special case arises when the subpopulations of
the finite normal mixture model share the same variance:
the original MLE is consistent in this case \citep{chen2017}.
Recently, \cite{LIU201929} showed that the MLE is
consistent under finite mixtures of location-scale distributions with a
structural scale parameter.

We are interested in finite mixtures of Gamma distributions.
These models have applications to data on the price
of commercial products,  the cost of insurance, household income, and so on:
see \cite{liu2003testing}, \cite{wong2014test}, \cite{willmot2011risk}, and \cite{yin2019consistency}.
This paper will provide an example in which these models
provide meaningful subpopulation structures for household income data.

Somewhat surprisingly, the likelihood function of the finite Gamma mixture
model is unbounded \citep{Chen2016}. This leads to the failure of the MLE
for this model, as for the finite normal mixture model.
Given the consistency results for the latter model,
it would be interesting to know if the MLE (without penalty)  is consistent when
the subpopulation distributions of the former model
have a {\it structural} shape or  scale parameter.
We prove that this is the case when the shape parameter is structural.

In the next section, we present our main results, showing that
the MLE under the two-parameter finite
Gamma mixture model is consistent when the shape parameter is structural.
In Section \ref{sec3},
we show that the MLE of the structural shape parameter almost surely
falls into a compact interval [$\tau, \Delta$]. This crucial result paves the
way for the final proof of the consistency of the MLE.
In Section \ref{sec4}, we prove the consistency conclusion.
In Section \ref{sec5}, we supplement the theoretical proof with
simulation experiments to numerically demonstrate the consistency
of the MLE.
In Section \ref{sec6}, we fit the income
data set with finite mixtures of Gamma distributions with various orders.
We find that a model of order three or
four provides a good fit, and the fitted models also suggest
subpopulation structures.
Section \ref{sec7} provides a discussion.

\section{Properties of the Gamma distribution and the finite Gamma mixture model}
\label{sec2}

\subsection{Preparation}
The Gamma distribution is the two-parameter
distribution for which the density function is given by
\[
 f(x; r, \theta)
 =\frac{x^{r-1}\exp{(-x/\theta)}}{\theta^r \Gamma(r)}
\]
over $x>0$ with shape parameter $r$ and scale parameter $\theta$.
The well-known Gamma function is defined, for all $ r > 0$, by
\[
\Gamma (r) = \int_0^\infty  x^{r-1} \exp(-x) dx.
\]
The parameter space of the two-parameter Gamma distribution is given by
\[
\Omega=\{(r, \theta): 0 < r < \infty, 0 < \theta < \infty \}
= \mathbb{R}^+ \times \mathbb{R}^+.
\]
The density function of this distribution is smooth and has nice analytical properties,
such as a convenient moment generating function.
However, if the shape parameter $r \in (0, 1)$, then the density function is unbounded.
This is evident when $r = 0.5$ and $\theta = 1$. In this case, we have
for some constant $C > 0$,
\[
f(x; r=0.5, \theta = 1) = C  x^{-.5}\exp( - x)  \to \infty
\]
as $x \to 0_+$.
This property complicates some of the technical discussion.
Since a smooth and one-to-one data transformation does
not lose any information,
it is convenient to introduce a  logarithm transformation to simplify the
presentation.

Suppose $X$ is a Gamma distributed random variable, and let $Y=\log{X}$.
The density function of $Y$ for the same parameter space
$\Omega$ becomes
\begin{equation}
\label{loggammadensity}
g(y; r, \theta)
=
\frac{1}{\Gamma(r)}\exp{\{r(y-\log{\theta})-\exp{(y-\log{\theta})}\}}
\end{equation}
for $y \in \mathbb{R}$.
Clearly, this density function in $y$ is bounded for any given $r$ and $\theta$.
Given a random sample from a finite Gamma mixture distribution,
a logarithm transformation will lead to a random sample from
a finite log Gamma mixture distribution.
The MLE of $G$ based on data in $X$ is the same as the MLE based on
data in $Y$.

For a finite mixture of log Gamma distribution
of order $m$ where the subpopulation distributions share an equal
shape parameter $r$ of unknown value the density function is
\be
\label{eq3.1}
g(y;  r, G)=
\int_{\mathbb{R}^+} g(y; r,\theta)\,dG(\theta)
=
\sum_{j=1}^m \alpha_j g(y; r, \theta_j).
\ee
In this case, $g(y; r, \theta)$ is the subpopulation/component density function, and
$(\alpha_1,\alpha_2, \ldots,\alpha_m)$ are the mixing proportions.
The mixing distribution can be presented by its cumulative distribution
function
\[
G(\theta)=\sum_{j=1}^m \alpha_j \mathbb{I}\{\theta_j\leq\theta\},
\]
where $\mathbb{I}(\cdot)$ is an indicator function.
The parameter $\theta_j$  is a support point of the mixing distribution $G(\theta)$.
We may also write $G(\theta)$  as shorthand for
$G=\sum_{j=1}^m \alpha_j \{\theta_j\}$.

In the above setting, the $G$ mixes only the scale parameter
and leaves $r$ as a common shape parameter.
For this reason, we call  $r$  a structural parameter;
its parameter space is $\mathbb{R^+}$.
We denote the space of all mixing distributions with at most $m$ supports in
$\mathbb{R^+}$ as
\begin{equation}
\label{mixing.space.of.scale.for.common.shape}
\mathbb{G}_m
=
\big \{G: G=\sum_{j=1}^m \alpha_j \{\theta_j\};
	\alpha_j \in [0,1], \theta_j \in \mathbb{R}^+; \sum_{j=1}^m \alpha_j=1;  j=1, \ldots, m \big\}.
\end{equation}
Note that $\mathbb{G}_m$ permits $\alpha_j = 0$ or equal $\theta_j$ values.
We are interested in the Gamma mixture model with
the parameter space of $(r, G)$ being $\mathbb{R}^+ \times \mathbb{G}_m$.

\subsection{Finite expectation and extended  Glivenko--Cantelli theorem }

The finite Gamma mixture model with a structural shape parameter has
some nice properties that are easy to verify. They are given below for
subsequent reference.

\begin{lemma}
\label{Expectation}
Let $X$ be a random variable with finite Gamma mixture distribution $f(x; r^*, G^*)$
where $G^*=\sum_{j=1}^m \alpha_j^* \{\theta_j^* \}$ for some $m$ and $Y = \log (X)$.
Let the density function of $Y$ be $g(y; r^*, G^*)$ as in \eqref{eq3.1}.
Then the expectations of $Y$, $\exp{(Y)}$, and $\log \{g(Y; r^*,G^*)\}$ exist and are finite.
\end{lemma}

\begin{proof}
The moment generating function of $Y$ is given by
\[
M_Y (t)
=\bbE^* \{ \exp{(t Y)}\}
=\sum_{j=1}^m\big \{\alpha_j^*\theta_j^{*t}\Gamma(r^*+t)/\Gamma(r^*)
\big \}
\]
which is well defined for $t > - r^*$.
Since the range of this function contains 0 as an interior point,
all the moments of $Y$ are finite.
Further, $\bbE^*[ \exp(Y)] = M_Y(1) < \infty$ so $\exp(Y)$ has finite expectation.

As a function of $y$, $g(y; r^*,G^*)$ has a finite upper bound.
Hence, $\bbE^*\{ \log{ g(y;r^*,G^*)} \} < \infty$.
In addition,
\ba
\log {g(Y; r^*,G^*)}
& \geq &
\log{\big \{\alpha_1^*g(Y;r^*,\theta_1^*)\big \}}\\
&=&
\log{\alpha_1^*} - \log{\Gamma(r^*)}+r^* Y -
r^*\log{\theta_1^*}-(\theta_1^{*})^{-1} \exp{(Y)}.
\ea
Clearly, every term on the right-hand side has a finite expectation.
Hence,  $\bbE^*\{ \log{ g(y;r^*,G^*)} \} > - \infty$.
This completes the proof.
\end{proof}

\begin{lemma}
\label{Glivenko-Cantelli}
In the setting of Lemma \ref{Expectation},
let $M=\sup_y {g(y; r^*,G^*)}$.

For any fixed positive number $\delta$, we have
\begin{equation}
\label{unifom1}
\sup_u\frac{1}{n}\sum_{i=1}^n\mathbb{I}(|y_i-u|<\epsilon)
<
2M\epsilon+\delta
\end{equation}
uniformly in $\epsilon$ almost surely.
\end{lemma}

\begin{proof}
Let $F_n(y) =n^{-1}\sum_{i=1}^n\mathbb{I}(y_i\le y)$ be
the empirical distribution function and $F(\cdot)$ be the distribution function of
$Y$.
It can be seen that
\[
\begin{aligned}
&\sup_{u}|n^{-1}\sum_{i=1}^n\mathbb{I}(|y_i-u|<\epsilon)-\mathbb{P}(|Y-u|<\epsilon)|\\
&\le\sup_{u}|F_n(u+\epsilon)-F_n(u-\epsilon)-F(u+\epsilon)+F(u-\epsilon)|\\
&\le
2 \sup_{u}|F_n(u) -F(u)| \to 0
\end{aligned}
\]
almost surely, by the Glivenko--Cantelli theorem, as $n \to \infty$.
That is,
\[
n^{-1}\sum_{i=1}^n\mathbb{I}(|y_i-u|<\epsilon)
\to \mathbb{P}(|Y-u|<\epsilon)
\]
almost surely and uniformly for $u$.
Consequently, uniformly for all $u$ and any small $\delta>0$,
\[
\sup_{u}\frac{1}{n}\sum_{i=1}^n\mathbb{I}(|y_i-u|<\epsilon)
<
\mathbb{P}( |Y-u| < \epsilon )+\delta
\]
almost surely.
Because $\mathbb{P}(|Y-u|<\epsilon)\le 2M\epsilon$, we further have
\[
\sup_{u}
\frac{1}{n}\sum_{i=1}^n\mathbb{I}\left(|y_i-u|<\epsilon\right)
<2M\epsilon+\delta,
\]
almost surely.
This completes the proof.
\end{proof}

The above lemma slightly extends the Glivenko--Cantelli theorem.
For any interval of length $\epsilon$, the proportion of a random sample it
contains is nearly uniformly bounded by $O(\epsilon)$.

\subsection{Inequalities}
The following lemma \citep{Xinli2007} gives a useful property of the Gamma function.

\begin{lemma}
\label{Gamma.bound}
When $r >1$, we have
\[
    \frac{r^{r-\gamma}}{e^{r-1}} < \Gamma(r) < \frac{r^{r-1/2}}{e^{r-1}},
\]
where $\gamma=0.577215\cdots$ is the Euler--Mascheroni constant.
When $0 < r < 1$, the left inequality holds, but the right inequality is reversed.
\end{lemma}

The following lemma gives altered versions
of two results from \cite{Chen2016}. We will not repeat the settings.

\begin{lemma}
\label{lemma.bound1}
(a)
For any $r>0$ and $\theta>0$, the density function of  $Y$ satisfies
\be
  \label{lemma.bound1.a}
\log{g(y;r,\theta)}\le \gamma\log{r}
\ee
where $\gamma < 1$ is the Euler--Mascheroni constant.

(b)
Let $\epsilon_r=\sqrt{2}\log{r}/\sqrt{r}$.
When $r>20$ and $|y-\log{(r\theta)}|>\epsilon_r$, we have
\be
  \label{lemma.bound1.b}
\log{g(y; r, \theta)} \leq \gamma( \log{r}-\log^2{r}).
\ee
\end{lemma}

\begin{proof}
It can easily be seen that
 \[
\log{g(y; r, \theta)}
= - \log{\Gamma(r)}+r (y-\log{\theta})-\exp{(y-\log{\theta})}.
\]
As a function of $y$, the above function attains its maximum at $y=\log{(r\theta)}$.
Hence,
\[
\log{g(y;r,\theta)} \le \log{g(\log{(r\theta)} ; r, \theta)}
=
-\log{\Gamma(r)} + r\log{r} - r.
\]
Recall Lemma \ref{Gamma.bound}, which gives
a lower bound on the log gamma function:
\be
\label{log.gamma.bound}
    \log{\Gamma(r)} \ge (r -\gamma) \log{r} - r + 1
\ee
for any $r>0$.
Applying this bound to the upper bound of $\log{g(y; r,\theta)}$, we have
\[
    \log{g(y; r, \theta)}\le \gamma \log{r}.
\]
This proves \eqref{lemma.bound1.a}.

Next, we prove \eqref{lemma.bound1.b} through a side result.
Let $t_0$ be any positive constant between 0 and 1.
We have
\begin{equation}\label{h(t)}
 h(t) = \exp{(t)}-t-1 \geq (1/3) t_0^2
\end{equation}
for $|t| \ge t_0$.
By expanding $h(t)$ to the quadratic term, we can see
that the inequality is true when $t > 0$.
Note that $h(t)$ decreases monotonically when $t \in (-\infty, 0)$.
Hence, the inequality is true if $h(- t_0) \geq (1/3) t_0^2$ for any $t_0 \in  (0, 1)$.
For any $t \in (0, 1)$, applying Taylor's expansion gives
\[
h(-t) - (1/3) t^2  = (1/6)(t^2 - \tilde t^3) \geq 0
\]
where $\tilde t \in (0, t)$.
This completes the proof that $h(t) \geq (1/3) t_0^2$.

A slight rearrangement of the terms gives
\[
\log{g(y;r, \theta)}
= - \log{\Gamma(r)}+r \log{r}+r(y-\log{(r\theta)}) - r\exp{(y-\log{(r\theta)})}.
\]
Applying  bound \eqref{log.gamma.bound}
to the log gamma function again, we have
\be
\label{loggy}
    \log{g(y; r, \theta)}\le \gamma\log{r}-1-r h(y - \log (r\theta)).
\ee
Let $t=y-\log{(r\theta)}$ and $t_0 = \epsilon_r=\sqrt{2}\log{r}/\sqrt{r}$.
The restriction $r>20$ implies that $0< t_0 <1$.
Applying the inequality for $h(t)$ with this $t_0$ we obtain
\ba
\log{g(y; r,\theta)}
&\leq \gamma \log{r}-1-r h(t)\\
&\leq \gamma \log{r}- \frac{2}{3}\log^2{r}\\
&\leq \gamma (\log{r}-\log^2{r}),
\ea
for all $r > 20$.
This completes the proof of \eqref{lemma.bound1.b}
and the proof of the lemma.
\end{proof}

The two inequalities in this lemma give us information about the log density function.
First, although this function is bounded in $y$, the upper bound can
be arbitrarily large when $r$ is very large. Hence, when the parameter space for $r$
has a finite upper bound, the density functions in terms of $g(y; r, \theta)$
have a uniform upper bound. This trivially implies the consistency of the MLE
of $G$ when a valid upper bound is placed on $r$.
Second, the density function peaks at $y = \log (r\theta)$,
but its value decays quickly at a quadratic rate in $\log r$ as
$y$ diverges from this value.

\section{Range of the MLE of the structural shape parameter $r$}
\label{sec3}
Suppose we have an IID sample  $y_1, y_2, \ldots, y_n$
from a finite log Gamma mixture distribution specified by \eqref{general.mixture}
in which $r$ is structural and $G \in \mathbb{G}_m$ for specified $m$.
The log likelihood function of the mixing distribution $G$ is given by
\begin{equation}
\label{3.1.likelihood}
\ell_n(r, G) = \sum_{i=1}^n \log{g(y_i; r, G)}.
\end{equation}
The MLE of $( r, G)$ is some $(\hat r, \hat G)$ such that
\begin{equation}
\label{definition.of.MLE.3.1.likelihood}
\ell_n(\hat r, \hat G)
=
\sup\{\ell_n(r, G): ( r, G) \in \mathbb{R}^+ \times \mathbb{G}_m\}.
\end{equation}
We have implicitly assumed that the MLE is the maximum
point of the likelihood among the mixing distributions with at most $m$
distinct support points.

When the shape parameter $r$ is confined in a compact finite
interval, we can easily show that the constrained MLE is consistent.
Hence, if feasible, a way to prove the consistency of the MLE is to show
that $\tau \leq \hat r \leq \Delta$ almost surely for some positive
constants $\tau$ and $\Delta$, where $\hat r$ is
the MLE of the structural shape parameter.
This strategy was used by \cite{chen2003tests} in the context of the finite normal mixture model;
for the finite Gamma mixture model the proof
is more complicated. In this section and the next, we will:
\begin{itemize}
\item[(a)]
show that there exist a sufficiently small positive constant $\tau > 0$ and
a sufficiently large positive constant $\Delta > 0$ such that
$\tau \leq \hat r \leq \Delta$ almost surely;
\item[(b)]
verify that the sufficient conditions presented in \cite{chen2017}
are satisfied by the finite Gamma mixture model with a reduced parameter
space $[\tau, \Delta] \times \mathbb{G}_m$ for any
$0 < \tau < \Delta < \infty$.
\end{itemize}

\begin{lemma}
\label{lemma.shape}
Assume that we have a set of IID observations from the finite Gamma mixture distribution.
Further, assume that the true mixing distribution is given by
\(
G^* = \sum_{j=1}^{m^*} \alpha_j^* \{\theta_j^*\}
\)
for some $\alpha^*_j > 0$ and the true structural parameter value $r^*$,
the distinct support points $\theta^*_j$, and $m^* \leq m$.
Let $(\hat r, \hat{G})$ be a global maximum point of
the likelihood function $\ell_n(r, G)$ over $\mathbb{R}^+ \times \mathbb{G}_m$.
There exist a sufficiently small constant $\tau>0$ and a sufficiently large
constant $\Delta$ such that as $n \to \infty$,
$\{\tau \le \hat{r} \le \Delta\}$ almost surely.
\end{lemma}

\begin{proof}
Recall that $g(y; r, G)$ is the density function of $Y = \log X$, which is a finite
mixture of log gamma distributions.
By inequality \eqref{lemma.bound1.a} in Lemma \ref{lemma.bound1},
we have $\sup_y g(y; r, G) \leq \gamma \log r$.
Hence,
\[
\sup_{0 <r<\tau} \ell_n(r, G)  < n \gamma \log{\tau}
\]
for any positive and small constant $\tau$.

In Lemma \ref{Expectation}, where the
distribution of $Y$ was assumed to be $g(y; r^*, G^*)$, we showed that
$\bbE^*\{ \log{g(Y; r^*, G^*)}\}$ is finite.
By the strong law of large numbers, we have
\[
\ell_n(r^*, G^*)= n \bbE^* \{\log{g(Y; r^*, G^*)}\} +o(n).
\]
Hence,
\[
\sup_{0<r<\tau} \ell_n(r,G) - \ell_n(r^*, G^*)
<
n\big \{\gamma\log{\tau} -\bbE^* [\log{g(Y; r^*, G^*) }] \big \}+o(n).
\]
Clearly,
$\gamma \log{\tau} - \bbE^*  \{ \log{g(Y; r^*, G^*)}\} < 0$
when $\tau$ is sufficiently small.
Therefore,
\[
n\big \{\gamma\log{\tau} -\bbE^* \log{g(Y; r^*, G^*)}\big \}+o(n)
< 0
\]
almost surely.
Hence,
\[
\sup_{0<r<\tau} {\ell_n(r, G)}-\ell_n(r^*,G^*) < 0
\]
almost surely. In other words,
\[
{\ell_n(r, G)}<\ell_n(r^*, G^*)
\]
uniformly for $r < \tau$ almost surely.
Based on this, we
conclude that almost surely, the MLE of $r$ will not be in the region
$(0, \tau)$.
This proves the first inequality of the lemma.

We now show that $\{\hat{r}\le \Delta\}$ almost surely for sufficiently large $\Delta$.
We first observe that
\[
g(y; r, G)
=
\sum_{j=1}^m \alpha_j g(y; r, \theta_j)
\le
\max_{1\le j \le m} g(y; r, \theta_j).
\]
Recall (see \eqref{lemma.bound1.b}) that for any $G$ and $r$,
when $r>20$ and $|y-\log{(r\theta_{j})}| >  \sqrt{2}\log{r}/\sqrt{r}$,
we have
\[
\log g(y; r, \theta) \le \gamma (\log{r}-\log^2{r}),
\]
and the upper bound is very small when $r$ is large.
For $y$ such that $|y-\log{(r\theta_{j})}|>  \sqrt{2}\log{r}/\sqrt{r}$
for all $j$ in $1, 2, \ldots, m$, we have
\be
\label{eq4.2}
\log{g(y; r, G)}
\le
\log{\{\max_{1\le j\le m}g(y; r, \theta_j)\}}\le\gamma(\log{r}-\log^2{r}).
\ee
For convenience, let  $\mu_j = \log{(r\theta_{j})}$
and $\epsilon_r =  \sqrt{2}\log{r}/\sqrt{r}$.
For $r >20$, applying \eqref{lemma.bound1.a}
and \eqref{eq4.2}, we have
\ba
\log \{g(y; r,G)\}
&\le &
\gamma\log{r} - \{ \gamma\log^2{r}\} \ind (\min_{1\le j\le m} |y-\mu_j|\ge\epsilon_r)\\
&= &
\gamma\log{r}- \gamma\log^2{r}+\{ \gamma\log^2{r}\} \ind (\min_{1\le j\le m}|y-\mu_j|<\epsilon_r)
\\
&\le &
\gamma \big  \{
\log{r}-\log^2{r} + \log^2{r}\sum_{j=1}^m \ind (|y-\mu_j|<\epsilon_r)
\big \},
\ea
where $\ind (\cdot)$ is the indicator function.

Applying this inequality to $y_1, y_2, \ldots, y_n$ in the likelihood function, we obtain
\bea
\label{q9}
\ell_n(r, G)
& = &
\sum_{i=1}^n \log \{g(y_i; r,G)\}
\nonumber \\
& \le &
n\gamma\log{r}-n\gamma \log^2{r} +\gamma\log^2{r}\sum_{j=1}^m
 \Big \{ \sup_{\mu_j} \sum_{i=1}^n\mathbb{I}(|y_i-\mu_j|<\epsilon_r)\Big \}
 \nonumber \\
 &\le &
 n \gamma \big ( \log{r} - \log^2{r} \big )
+ m  n \gamma  \log^2{r}
 \sup_{u}
\Big \{ n^{-1} \sum_{i=1}^n\mathbb{I} \big (|y_i-u|<\epsilon_r \big ) \Big \}.
\eea
By the upper bound in Lemma \ref{Glivenko-Cantelli}, we have
\[
 \sup_{u}
 \Big \{ n^{-1} \sum_{i=1}^n\mathbb{I} \big (|y_i-u|<\epsilon_r \big )
 \Big \}
< 2M^* \epsilon_r +  \delta
\]
almost surely for any fixed $\delta > 0$, where $M^*$ is the upper bound
on the true density function, which is finite.
Furthermore, by choosing a sufficiently large $\Delta$,
we can ensure that $2M^*\epsilon_r < \delta$ uniformly over $r > \Delta$.
Hence,
\[
 \sup_{u}
 \Big \{ n^{-1} \sum_{i=1}^n\mathbb{I} \big (|y_i-u|<\epsilon_r \big )
 \Big \}
< 2 \delta
\]
almost surely for any $\delta > 0$.
By choosing a sufficiently small $\delta$ and the corresponding $\Delta$,
we have that uniformly over $r > \Delta$,
\[
\ell_n(r, G)
\leq
n\gamma \{ \log r - (1- m \delta) \log^2 r \}
< n \bbE^*\{ \log g(Y;r^*, G^* )\}.
\]
Further, because
\[
\ell_n(r^*,G^* ) = n [ \bbE^*\{ \log g(Y; r^*,G^* )\} + o(1)] ,
\]
we obtain
\[
\sup_{r > \Delta} \ell_n(r,G)  < \ell_n(r^*, G^*)
\]
almost surely. This is the second inequality of the lemma, and
this completes the proof.
\end{proof}

\section{Consistency of the MLE for a restricted structural shape parameter}
\label{sec4}
We have accomplished task (a):
 the MLE of the structural shape parameter is almost surely in a finite interval
$[\tau, \Delta] \subset {\mathbb R}^+$.
The consistency problem for the finite Gamma mixture model
with structural $r$ has therefore been reduced to the problem
where the parameter space of $(r, G)$ is  $[\tau, \Delta] \times \mathbb{G}_m$.

To conveniently discuss the consistency of the  MLE $(\hat r, \hat G)$,
we introduce a distance \citep{KW1956} on the parameter space of $(r, G)$.
For any shape parameter values $r_1, r_2$ and mixing distributions $G_1, G_2$, let
\begin{equation}
\label{DKW}
D_{KW}((r_1, G_1),  (r_2, G_2))
=
|\arctan( r_1) - \arctan (r_2) |
+ \int_{ \mathbb{R}^+} |G_1(\theta)-G_2(\theta)| \exp(-\theta)\,d\theta.
\end{equation}
It can be seen that
\[
D_{KW}((r_1, G_1),  (r_2, G_2))
\leq
\pi + \int_{ \mathbb{R}^+} \exp(-|\theta|)\,d\theta
= \pi + 1
\]
for any $(r_1, G_1)$ and $(r_2, G_2)$.
That is, the space of mixing distributions is totally bounded.
It is important to note that a totally bounded and closed space is compact.

Another important property is that $D_{KW}((r_k, G_k), (r^*, G^*)) \to 0$
if and only if $r_k \to r^*$ and $G_k \to G^*$ in distribution/measure.
Hence, the consistency of the MLE $(\hat r, \hat G)$
can be conveniently interpreted as $D_{KW}((\hat r, \hat G), (r^*, G^*)) \to 0$
almost surely as the sample size $n \to \infty$.
We now state our main result.

\begin{theorem}
\label{theorem1}
Assume that we have a set of IID observations $y_1, \ldots, y_n$ from \eqref{general.mixture}
where $G^* = \sum_{j=1}^{m^*} \alpha_j^* \{\theta_j^*\}$ with $\alpha^*_j > 0$
and $m^* \leq m$.
Then the MLE of $(r, G)$ defined in \eqref{definition.of.MLE.3.1.likelihood},
$(\hat r, \hat{G})$, is a consistent estimator.
That is, almost surely as $n \to \infty$
\[
D_{KW}((\hat r, \hat{G}), (r^*, G^*)) \to  0.
\]
\end{theorem}

The proof will be presented in the next two subsections.

\subsection{Four conditions}
There is a rich literature on the consistency of the MLE under mixture models.
According to \cite{chen2017}, there are three ways to
establish the consistency of the MLE: the approaches of \cite{KW1956}, \cite{Redner1981},
and \cite{Pfanzagl1988}. They give
similar but not completely equivalent results.
Our proof follows the approach of \cite{KW1956}.

To avoid confusion in notation between the specific finite Gamma mixture
and the general mixture model, we use $f(x; \psi)$ for
the density function of the component distribution for the general mixture.
A general mixture model has the density function
\be
\label{mixture.general}
f(x; G) = \int_\Psi f(x; \psi)\,dG(\psi)
\ee
for $G \in \mathbb{G}$.
Note that the finite mixture model is a special case where
$ \mathbb{G}$ is reduced to  $\mathbb{G}_m$.
The parameter space of $\psi$ is $\Psi \subset \mathbb{R}^d$
for some positive integer $d$.

\cite{KW1956} observed that under some conditions,
the consistency of the MLE under a mixture model based on
IID observations is reduced to the general consistency problem
discussed in \cite{Wald1949}.
The conditions for consistency given in
\cite{KW1956} are detailed but hard to comprehend.
\cite{chen2017} streamlined their conditions and
replaced them by the following four high-level conditions
applicable to \eqref{mixture.general}:

\begin{itemize}
\item[A1]
Identifiability: Let $F(x; G)$ be the cumulative distribution function of $f(x; G)$.
The mixture model is identifiable, i.e.,
\[
F(x; G_1)=F(x; G_2)
\]
for all $x$ implies $G_1=G_2$.

\item[A2]
Finite Kullback--Leibler Information: Let the true mixing distribution be $G^*$
and for any subset $B$ of mixing distributions, define
\[
f(x; B)=\sup_{G \in B} f(x; G).
\]
Let
$B_\epsilon(G)=\{G': D_{KW}(G, G')<\epsilon\}$ be an open ball of radius $\epsilon$
centered at $G$.
For any $G \ne G^*$, there exists an $\epsilon>0 $ such that
\[
\bbE^*\big [ \log{\{f(X; B_\epsilon(G))/f(X; G^*)\}}\big ]^+ <\infty.
\]
The expectation operator $\bbE^*(\cdot)$ is taken over $f(x; G^*)$.

\item[A3]
Continuity: The component parameter space $\Psi$ is a closed set.
For all $x$ and any given $G_0$, we have
\[
\lim_{G \to G_0} f(x; G) = f(x; G_0).
\]

\item[A4]
Compactness: The definition of the mixture density $f(x; G)$ in $\mathbb{G}$
can be continuously extended to a compact space $\bar{\mathbb{G}}$
while retaining the validity of Condition A2.
\end{itemize}

The above sufficient conditions are applicable
to the nonparametric MLE $\hat G_n$ for general mixture models.
They are equally applicable to finite mixture models when $\mathbb{G}$
is reduced to $\mathbb{G}_m$ and $G^* \in \mathbb{G}_m$.
If we regard the shape parameter $r$ as part of the mixing, then
the mixing distribution degenerates in this aspect,
but the conclusion remains applicable.
To prove Theorem \ref{theorem1} we may show that
Conditions A1--A4 are satisfied.

We have already shown that $\hat r$ is almost surely part of
$[\tau, \Delta] \subset {\mathbb R}^+$. In Section~\ref{conditions-reduced}
we show that Conditions A1--A4 are satisfied for the finite Gamma mixture model
with structural shape parameter $r$ and parameter space
$[\tau, \Delta] \subset {\mathbb R}^+$.

\subsection{Conditions A1--A4 on the reduced parameter space}\label{conditions-reduced}

The finite Gamma mixture model with structural shape parameter $r$
is a special model: its bivariate mixing distribution degenerates
in its structural elements $r$. The identifiability of this model is implied by
the identifiability of the general finite Gamma mixture.
\cite{Henry1963} gave a set of sufficient conditions for the identifiability of
finite mixture models, and
the finite mixture of the two-parameter Gamma model satisfies
these conditions. Hence, identifiability holds, and
Condition A1 is verified.

Conditions A2 and A3 are part of A4:
verifying A4 verifies A2 and A3 at the same time.
The first task is to extend the component parameter space of $\theta$
from $\mathbb{R^+}$ to its closure $[0, \infty]$.
For any given $r \in [\tau, \Delta]$, it can easily be seen that
\[
\lim_{\theta \to 0+} g(y; r, \theta)
=
\lim_{\theta \to \infty}g(y; r, \theta)
=0
\]
for all $y$.
Therefore, for any given $r \in [\tau, \Delta]$, we can extend the
subpopulation parameter space of $\theta$ to $[0, \infty]$
by defining
\[
g(y; r, 0) = g(y; r, \infty) =0
\]
for all $y \in \mathbb{R}$.
With this extension, the density function $g(y; r, \theta)$
remains continuous in $r$ and $\theta$ on $[\tau, \Delta] \times [0, \infty]$.

Next, let
\[
\bar{\mathbb{G}} = \{ \rho G: G \in \mathbb{G}_m, \rho \in [0, 1] \}
\]
so that for any $r \in [\tau, \Delta]$ and $\bar G \in \bar{\mathbb{G}}$, we define
\[
f(y; r, \bar G)  = \int_0^\infty g(y; r, \theta) d \bar G
=
\rho  \int_0^\infty g(y; r, \theta) dG
= \rho f(y; r, G).
\]
For the usual case of $\bar G \in \mathbb{G}_m$, $\rho = 1$.
By Helly's selection theorem \citep{vanderVaart2000},
each sequence of probability measures $G_k: k=1, 2, \ldots$ has a converging subsequence
with a limit $\bar G$ in the terms of $D_{KW}( G_k, \bar G) \to 0$ allowing
$\bar G \in \bar{\mathbb{G}}$.
Equipped with the distance $D_{KW}(\cdot, \cdot)$, $\bar{\mathbb{G}}$
is compact.
Moreover, the density function $f(y; r, \bar G)$ is continuous. Hence,
Condition A3 is satisfied.

For each $y$, and given $r \in [\delta, \Delta]$,
we note that $g(y; r, \theta)$ has a finite upper bound in view of its
expression given in \eqref{loggammadensity}.
By one version of the definition of convergence in measure \citep{vanderVaart2000},
we have
\[
g(y; r, G_k) = \int_0^\infty g(y; r, \theta) d G_k
\to
\int_0^\infty g(y; r, \theta) d \bar G
\]
as $G_k \to \bar G$ in distribution.
This verifies Condition A3 after $\mathbb{G}_m$ is compactified.

Now we are ready to verify that Condition A2 is also satisfied on the space of
$(r, G) \in[\tau,\Delta] \times \bar{\mathbb{G}}_m$.
By \eqref{lemma.bound1.a}, we have
$\log{ g(y; r,G) } \le \gamma \log \Delta$ when $r \in [ \tau, \Delta ]$.
Hence, for any $(r, G) \in [ \tau, \Delta ]  \times  \bar{\mathbb{G}}_m$
and any constant $\epsilon>0$, let
\[
B_{\epsilon}(r, G) = \{(r',G'): D_{KW}\big( (r,G), (r',G') \big ) <\epsilon\}
\]
be an open ball of radius $\epsilon$ centered at $(r, G)$.
We have
\[
g( y; B_{\epsilon}(r, G) )
=
\sup_{ (r',G') \in B_{\epsilon}(r, G)}  g(y; r', G' ) \leq \gamma \log \Delta.
\]
Since $\bbE^*\log \{ g(Y; r^*, G^*) \} $ is finite, we get
\[
\bbE^* \log \{ g(Y; B_{\epsilon}(r,G)) / g(Y; r^*,G^*) \}
\leq
\gamma\log\Delta- \bbE^*\log \{  g(Y;r^*, G^*) \} <\infty,
\]
where the expectation $\bbE^*(\cdot)$ is taken
assuming that $Y$ has the true distribution specified by $g(y; r^*, G^*)$.
Therefore,  we have
\[
 \bbE^*[ \log{\{  g(Y; B_{\epsilon}(r,G)) / g(Y; r^*,G^*) \}} ]^+ <\infty,
\]
for all $(r,G)\in[\tau, \Delta]\times \bar{\mathbb{G}}_m$
such that $(r,G)\neq(r^*,G^*)$.
Thus, Condition A2 is also satisfied on the compactified space
$[ \tau, \Delta ]  \times  \bar{\mathbb{G}}_m$.

We have now shown that Conditions A1--A4 are satisfied.
Therefore, the MLE over the space of $r \in [\tau, \Delta]$
and $G \in  \mathbb{G}_m$ is strongly consistent.
This completes the proof.

One minor remark concerns $(\hat r, \hat{G})$ when it is the
maximum point of $\ell_n(r, G)$ over the compactified space
$\mathbb{R}^+\times \bar{\mathbb{G}}_m$.
It appears that $\hat G$ could be a subdistribution with a total
probability below 1, i.e., $\hat G = \rho \tilde G$ for
some $\rho \in (0, 1)$ and  $\tilde G \in {\mathbb{G}}_m$.
This is impossible because it would give
\[
\ell_n(\hat r, \hat G)
=
\ell_n(\hat r, \rho \tilde G)
<
\ell_n(\hat r, \tilde G)
\]
implying that $(\hat r, \hat G) $ is not the MLE.
We need subdistributions to
use the compact property in the proof.

\section{Numerical computations and simulation experiments}
\label{sec5}

\subsection{EM algorithm}
The EM algorithm is the most popular numerical approach for finding
the MLE of the mixing distribution.
The algorithm is iterative: it updates the initial mixing distribution $G^{(0)}$
proposed by the user to obtain $G^{(k)}: k=1, 2, \ldots$.
It is well known that $\ell_n(G^{(k)})$ is monotonic, which leads to
the convergence property.
The properties of the EM algorithm for finite mixtures
have been thoroughly discussed in \cite{Wu1983}.
We follow \cite{Chen2016} and use a slightly
adapted EM algorithm in our simulation experiments.

In a finite mixture model, each observation $x_i$ may be regarded
as part of the complete observations $(x_i, z_i)$ on the $i$th sample unit.
In this setup, $z_i$ is an unknown latent variable;
given $z_i = j$, $x_i$ is a sample from the $j$th subpopulation,
$j=1, 2, \cdots, m$.
Hence, the  complete-data log-likelihood for the
finite Gamma mixture model with a structural $r$ is given by
\[
\sum_{j=1}^m \sum_{i=1}^n
\ind (z_i=j) \{\log{\alpha_j}+\log{f(x_i; r, \theta_j)}\}.
\]
To improve the finite-sample performance, one may modify the log-likelihood
function by adding an $O_p(1)$ term without altering the consistency conclusion.
We use
\[
\ell_c (r, G)
=
\sum_{j=1}^m \sum_{i=1}^n
\ind (z_i=j) \{\log{\alpha_j}+\log{f(x_i; r, \theta_j)}\}
+ \epsilon  \sum_{j=1}^m \log{\alpha_j}
\]
for some $\epsilon > 0$.
In our simulation and real-data experiments, we choose $\epsilon = 0.001$.
This strategy was first employed in the modified likelihood approach of
\cite{Chen1998} and has been widely adopted, e.g., \cite{Chen2016}.
Since the $z_i$'s are missing, one cannot estimate $(r, G)$ by the maximum
point of $\ell_c (r, G)$.  To overcome this obstacle,
the EM algorithm replaces $\ind (z_i = j)$
by its expected value.
We now discuss the two steps of the algorithm.

\textbf{E-step.}
Given a shape parameter $r^{(0)}$ and a mixing distribution $G^{(0)} \in \mathbb{G}_m$,
we find the conditional expectation of $ \ind(z_i=j) $ given the data:
\[
w_{ij}^{(0)}
= \bbE^{(0)} \{ \ind(z_i=j) | x_1, \ldots, x_n \}
= \frac{ \alpha_j^{(0)}  f(x_i; r^{(0)},  \theta_j^{(0)}) }
{ \sum_{k=1}^m  \alpha_k^{(0)}  f(x_i; r^{(0)},\theta_k^{(0)}) }.
\]
Replacing $ \ind(z_i=j)$ by its conditional expectation given above, we obtain
\[
Q(r, G; r^{(0)}, G^{(0)})
=\sum_{j=1}^m \sum_{i=1}^n
w_{ij}^{(0)} \{\log{ \alpha_j + \log{f(x_i; r, \theta_j)} } \}
+
 \epsilon  \sum_{j=1}^m\log{\alpha_j}.
\]

\textbf{M-step.}
In this step, we maximize $Q(r, G; r^{(0)}, G^{(0)})$ with respect to
$r \in \mathbb{R}^+$ and $G \in \mathbb{G}_m$.
For $j=1, \ldots, m$, denote
\ba
\bar{w}_{j}^{(0)}&= &n^{-1} \sum_{i=1}^n w_{ij}^{(0)},\\
\bar{x}_{j}^{(0)} & = & \{n\bar{w}_{j}^{(0)}\}^{-1} \sum_{i=1}^n w_{ij}^{(0)}x_i,\\
\bar{y}_{j}^{(0)} & = &  \{n\bar{w}_{j}^{(0)}\}^{-1} \sum_{i=1}^n w_{ij}^{(0)}\log{x_i}.
\ea
We then get the expression
\bea
Q(r, G; r^{(0)}, G^{(0)})
&=&
n\sum_{j=1}^m
\big \{\bar{w}_j^{(0)} \big [ (r-1) \bar{y}_j^{(0)}
- (\bar{x}_j^{(0)}/\theta_j) -
    \log{\Gamma(r)}-r\log{\theta_j}\big ] \big \}
    \nonumber \\
 && +
\sum_{j=1}^m \{(n\bar{w}_j^{(0)}+\epsilon)\log{\alpha_j}\}.
    \label{objectfunction}
\eea
Note that the component parameters in  $Q(r, G; r^{(0)}, G^{(0)})$
are well separated.
Maximizing $Q(r, G; r^{(0)}, G^{(0)})$ with respect to $\alpha_j$
gives
\[
\alpha_j^{(1)} =( n\bar{w}_j^{(0)} + \epsilon)/(n + m\epsilon).
\]
The extra positive constant $\epsilon$ makes the above
iteration step numerically stable.

For each fixed $r$,  $Q(r, G; r^{(0)}, G^{(0)})$ is maximized with respect to $\theta_j$
when
\[
 \theta_j^{(1)} = {\bar{x}_j^{(0)}}/{r}.
\]
Replacing $ \theta_j$ by $ \theta_j^{(1)}$ in $Q(r, G; r^{(0)}, G^{(0)})$,
we find that the maximization solution in $r$ is given by
\[
 r^{(1)}
 =
 \mathop{\arg\max}_{r}
 \Big \{
 \sum_{j=1}^m
\big [ \bar{w}_j^{(0)}(\bar{y}_{j}^{(0)} - \log\bar{x}_{j}^{(0)}) \big ] r
+
\big [
r\log{r} - \log{\Gamma(r)}- r
\big ]
\Big \}.
\]
This is a single-variable function that can easily be solved.
Once $ r^{(1)}$ is obtained,  the updated mixing distribution is given by
\[
G^{(1)}=\sum_{j=1}^m \alpha_j^{(1)}\{\theta_{j}^{(1)}\}.
\]

Starting from the initial value $(r^{(0)}, G^{(0)})$,  the E- and M-steps give us $(r^{(1)}, G^{(1)})$.
Repeating these two steps, we get a sequence $(r^{(k)},G^{(k)}), k =1, 2, \ldots$.
The slightly modified log likelihood function has its value increased after each iteration.
We terminate the algorithm when the modified log likelihood value stabilizes;
in the simulations, we set the tolerance to $10^{-6}$.

\subsection{Simulation experiments}
We conducted simulation experiments to illustrate the consistency
properties of the MLE when $r$ is structural under a finite Gamma mixture model.
We generated data from the six Gamma mixture distributions specified in the following table.
Model I contains three mixtures of order $m=2$, and
Model II contains three mixtures of order $m=3$.
We selected distinct subpopulation scale parameter values for both models.

\begin{table}[h]
\begin{tabular}{ccc}
Model & Density function & $r$ \\ \hline
I & $ 0.4 f(x; r, 0.5)+0.6f(x; r, 5)$  & $0.5, 5, 50$\\
II & $  0.35 f(x; r, 0.5) + 0.55 f(x;  r,  2) + 0.1 f(x; r,  6)$ & $0.5, 10, 30$\\
\hline
\end{tabular}
\end{table}

From both models, we generated samples of sizes 60, 240, 960, and 3,840.
For each random sample, we used 30 sets of initial values to drive the EM algorithm.
The first 11 sets contain the true values of the model that generated
the sample and 10 randomly and mildly perturbed true values.
The remaining 19 sets are randomly generated, and
they can be quite different from the true value.
With these 30 initial values, up to 30 local
maxima are found for each random sample;
we took the  MLE to be the one with the highest $\ell_n$ value.
We performed $K=1,000$ repetitions for each mixture distribution
and sample-size combination.

We found the root mean square error (RMSE) for each parameter in these
six mixtures.
Let $\hat \psi$ be a generic parameter estimator and $\psi^*$ be the true value
of the corresponding parameter in the selected distribution.
The root mean square error (RMSE) of $\psi$ is
\[
 \mbox{RMSE}(\psi)
 =
 \Big \{
 K^{-1} \sum_{k=1}^K (\hat{\psi}^{(k)} - \psi^*)^2
 \Big \}^{1/2}.
\]
In this expression, we use the superscript $(k)$ for the estimate based on the $k$th sample,
and $\psi$ is generic notation for either the mixing proportions $\alpha_j$, the subpopulation
scale $\theta_j$, or the structural shape parameter $r$.
The simulation results are presented in Tables \ref{SimuResult1} and \ref{SimuResult2}.

We kept a record of the initial values (perturbed or random)
that led to the highest modified likelihood values.
Let $K_0$ be the number of times in the $1,000$ repetitions that
the highest values were obtained from the perturbed values.
We computed $\eta=K_0/K$ and report it in the last
column of Tables \ref{SimuResult1} and \ref{SimuResult2}.
The simulation results lead to the following two observations:

\begin{enumerate}
\item
As shown in the first column of the tables, for each $r$,
we increase the sample size in multiples of four from 60 to 3,840.
The increase often halves the RMSE of a parameter.
Overall, the RMSEs decrease markedly as the sample size increases.
These observations support the theoretical consistency conclusion.

There is one exception: $r=0.5$ under Model II.
Table~\ref{SimuResult2} shows that when the sample size increases from $60$ to $240$,
the RMSE for $\theta_3$ increases slightly.
This unexpected outcome may be attributed to the slow action of the asymptotic
when $r < 1$ and to the random nature of the simulation experiment.
Note that in this case the density function goes to infinity when $x$ approaches $0$.

\item
When $r = 0.5$, $\eta$ is usually below 50\%,
but it increases with the sample size.
This also indicates that in this case the asymptotic requires a
large sample.
If $r$ is large, the MLE is near the true parameter value
so it is usually located when the EM algorithm starts from the perturbed values.
\end{enumerate}

\begin{table}[H]
    \centering
    \caption{RMSEs of MLE based on data generated from Model I}
    \label{SimuResult1}
    \begin{tabular}{crcccccc}
    \\[-5mm]
    \toprule
    & $n$& $\alpha_1$ & $\alpha_2$  &  $r$   & $\theta_1$ & $\theta_2$  & $\eta$  \\
        \hline
    \multirow{4}{*}{$r$=0.5}
   &60&0.210 & 0.210 & 0.153 & 0.511 & 3.210 & 0.361 \\
    &240&0.113 & 0.113 & 0.063 & 0.321 & 1.296  & 0.413\\
    &960&0.050 & 0.050 & 0.028 & 0.120 & 0.455  & 0.431\\
    &3,840&0.025&0.025&0.014 & 0.059 & 0.216  & 0.453\\

    \hline
    \multirow{4}{*}{$r$=5}
   &60&0.067 & 0.067 & 1.083 & 0.104 & 0.982  & 0.822\\
   &240&0.033 & 0.033 & 0.491 & 0.053 & 0.501  & 0.792\\
   &960&0.016 & 0.016 & 0.228 & 0.026 & 0.240  & 0.821\\
   &3,840&0.008 & 0.008 & 0.116 & 0.013 & 0.123  & 0.865\\
 \hline
    \multirow{4}{*}{$r$=50}
    &60&0.063 & 0.063 & 10.74 & 0.091 & 0.910  & 1.000\\
    &240&0.031 & 0.031 & 4.690 & 0.045 & 0.454  & 1.000\\
    &960&0.016 & 0.016 & 2.283 & 0.023 & 0.228  & 1.000\\
    &3,840&0.008 & 0.008 & 1.164 & 0.012 & 0.116  & 1.000 \\
       \bottomrule
    \end{tabular}
    \end{table}

\begin{table}[H]
    \centering
    \caption{RMSEs of MLE based on data generated from Model II}
    \label{SimuResult2}
    \begin{tabular}{rrcccccccc}
    \\[-5mm]
    \toprule
    & $n$& $\alpha_1$ & $\alpha_2$  & $\alpha_3$ & $r$   & $\theta_1$ & $\theta_2$  &$\theta_3$  & $\eta$ \\

\hline
    \multirow{4}{*}{$r$=0.5}
   &60&0.251 & 0.285 & 0.384 & 0.387 & 0.444 & 2.742 & 5.188  &0.288\\
    &240&0.246 & 0.210 & 0.272 & 0.121 & 0.412 & 1.405 & 5.195  & 0.405\\
    &960&0.209 & 0.134 & 0.186 & 0.042 & 0.316 & 0.875 & 3.785  & 0.237\\
    &3,840&0.123&0.080&0.088&	0.016&	0.169&	0.502&	2.211&	0.578 \\

    \hline
    \multirow{4}{*}{$r$=10}
    &60&0.065 & 0.070 & 0.046 & 2.670 & 0.118 & 0.465 & 1.694  & 0.848\\
    &240&0.031 & 0.035 & 0.022 & 1.106 & 0.056 & 0.225 & 0.778  & 0.874 \\
    &960&0.017 & 0.017 & 0.011 & 0.519 & 0.027 & 0.108 & 0.377  & 0.910\\
   &3,840&0.008 & 0.009 & 0.005 & 0.261 & 0.014 & 0.055 & 0.193  & 0.890\\
\hline
    \multirow{4}{*}{$r$=30}

   &60&0.060 & 0.062 & 0.040 & 6.708 & 0.093 & 0.369 & 1.211 &  0.910\\
   &240&0.030 & 0.032 & 0.019 & 2.885 & 0.047 & 0.186 & 0.599  & 0.899\\
   &960&0.016 & 0.016 & 0.010 & 1.363 & 0.023 & 0.091 & 0.288  & 0.891  \\
   &3,840&0.007 & 0.008 & 0.005 & 0.693 & 0.012 & 0.046 & 0.147  & 0.890\\
       \bottomrule
    \end{tabular}
\end{table}

\section{Data example: Disposable income}
\label{sec6}

Finite Gamma mixture distributions are often used to model, for example,
insurance payments, household incomes, and the cost of medicine:
see \cite{liu2003testing}, \cite{wong2014test}, \cite{willmot2011risk}, and \cite{yin2019consistency}.
In this section, we illustrate the use of the finite Gamma mixture distribution
with a structural parameter for data on disposable income and expenditure.
We obtained the data
from the China Institute For Income Distribution \citep{CHIP13}.
They were collected in the fifth-wave survey in July and August 2014.
The CHIP13 data contain many attributes;
we analyze only the household income and expenditure.

The data set contains 17,244 records for disposable income, but
85 of them are either missing or nonpositive.
We exclude these records and fit a finite mixture model to the
remaining 17,159 observations. Table \ref{Summary.realdata} gives summary
statistics, and we make three remarks below:

\begin{enumerate}
\item
The maximum income  is about 35 times the mean income, and
the mean is much larger than the  median. Both features reflect the uneven
wealth distribution.

\item
Slightly over 10\% of the households have an income that is 50\%
above the 75th percentile.
The data are seriously skewed to the right.

\item
The majority of households (over 98\%) have an annual income
below 196 thousand yuan.
For the top $2\%$ the range is $[195, 959]$ thousand yuan.
\end{enumerate}
These characteristics suggest that the population
can be segregated into several homogeneous subpopulations.
A finite Gamma mixture distribution with a structural shape parameter
is a reasonable choice.

\begin{table}[H]
    \caption{Summary statistics for household income (in thousand yuan) }
\label{Summary.realdata}
\begin{center}
\begin{tabular}{ccccc}
    \toprule
   Minimum& 25th percentile &Median&75th percentile&Maximum\\
   0.0407&24.4149&42.4569&70.7991&1958.9434\\
   \hline
   Sample size&Mean&Standard deviation&Skewness&Kurtosis\\
   17,159&55.7077&53.5846&7.7611&187.3505\\
\bottomrule
\end{tabular}
\end{center}
\end{table}

\begin{figure}[H]
\centering
{
\includegraphics[width=14cm]{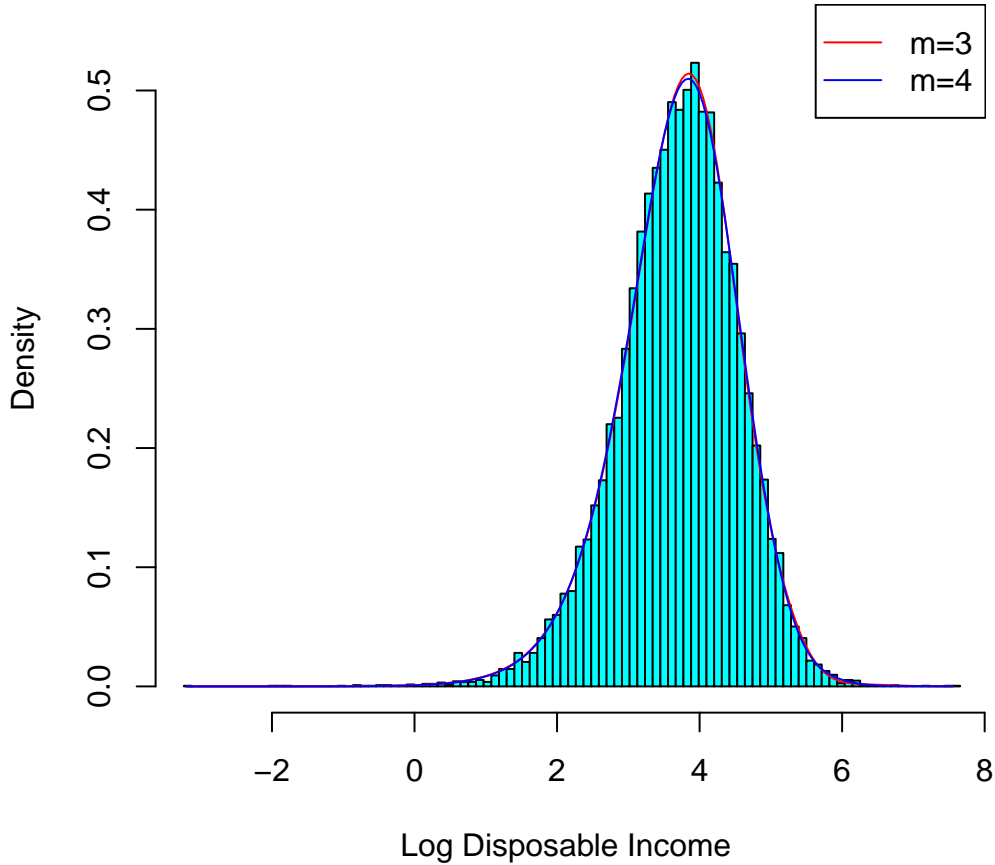}}
\hspace{0in}
\caption{Histogram of log disposable income
and density function of fitted log Gamma mixtures with order $m=3$ and $m=4$.}
\label{hist.and.fitted}
\end{figure}

We fit a set of  finite Gamma mixture models with a structural shape parameter $r$
and order $m=1, 2, \dots, 7$.
When $m$ is greater than 2, we first used 50 random initial values
to drive the EM algorithm and obtained up to 50 local maxima of the likelihood function.
We took the local maximum with the highest likelihood value as the tentative MLE.
Next, we created 10 initial values by perturbing the tentative MLE
and generated a further 19 randomly.
When the EM iteration stopped, we selected the estimate with the highest $\ell_n$ value
as the MLE.
For numerical stability, we adopted
the modified likelihood with $\epsilon = 0.001$.
Table \ref{MLE.realdata} gives the MLEs and corresponding likelihood values.
By the nature of the maximum likelihood,
$\ell_n (\hat G)$ increases as the order $m$ of the mixture models
increases. The size of the increment stalls at $m=7$.

Figure \ref{Q-Q_plot_of_disposable_income_with_various_order_m}
gives a QQ-plot for the fitted Gamma mixture models (log transformed),
together with the 45-degree line.
The data points nearly perfectly align with
the straight line when $m=3$ and $4$.
There is practically no room for further improvement by increasing $m$.

With a structural shape parameter $r$, the number of parameters
in the finite Gamma mixture model does not increase as quickly with $m$.
Although the mixture model is not regular, the Bayes information
criterion still provides some guidance.  For the current sample size,
a higher order of mixture model will be recommended  when the log-likelihood
is increased by 40.
With this general guidance, a finite Gamma mixture of order $m=3$
is recommended; $m=4$ is also acceptable.

The fitted finite Gamma mixture distribution of order $m=3$ suggests that about
75\% of the households have a low mean annual disposable income of $19,000$
yuan. A small percentage, 0.3\%, of the households have 10
times this value.
Setting $m=4$ changes the picture of
the low and medium-income households only slightly.
However, it separates a much smaller percentage, 0.04\%, of
super-rich households. They have nearly 30 times the
mean income of the low-income households.

\begin{table}[H]
\centering
\caption{MLEs of Gamma mixtures for disposable income}
\label{MLE.realdata}
\begin{tabular}{cccccccc}
\\[-5mm]
\toprule
$m$ &$1$ &$2$&$3$&$4$&$5$&$6$&$7$\\
\hline
$\hat{r}$&1.6945&2.0261&2.1841&2.2433 	&	2.3293 	&	2.3526 	&	2.4391 	\\
\hline
$\hat\alpha_1$&&0.9129&0.7436&0.6009 	&0.0008 	&0.0002 	&0.0002 	\\
$\hat\alpha_2$&	&0.0871&0.2531&	0.3740 	&	0.5486 	&	0.0008 	&	0.0009 	\\
$\hat\alpha_3$&	&&	0.0033 	&	0.0247 	&	0.4214 	&	0.5363 	&	0.0094 	\\
$\hat\alpha_4$&&&	&	0.0004 	&	0.0288 	&	0.4323 	&	0.5311 	\\
$\hat\alpha_5$&&&		&		&	0.0004 	&	0.0300 	&	0.4280 	\\
$\hat\alpha_6$&&		&		&		&		&	0.0004 	&	0.0301 	\\
$\hat\alpha_7$&	&	&		&		&		&		&	0.0004 	\\
\hline
$\hat\theta_1$&32.876&23.744&19.290&17.189&0.2408&0.0526 	&	0.0509 	\\
$\hat\theta_2$&&66.805&41.620&33.102&15.714&	0.4046 	&	0.4037 	\\
$\hat\theta_3$&&&190.82&78.326&30.811 	&	15.353 	&	5.0650 	\\
$\hat\theta_4$&&&&484.69 &	73.740 	&	30.271 	&	14.804 	\\
$\hat\theta_5$&&		&		&		&	466.97 	&	72.579 	&	29.496 	\\
$\hat\theta_6$&	&		&		&		&		&	462.15 	&	70.756 	\\
$\hat\theta_7$&&		&		&		&		&		&	449.30 	\\
$\ell_n$&-84,911&-84,542&-84,487 &	-84,478 &-84,468&-84,467&-84,467 	\\
\bottomrule
    \end{tabular}
\end{table}

\begin{figure}[H]
\centering
{
\includegraphics[width=6cm]{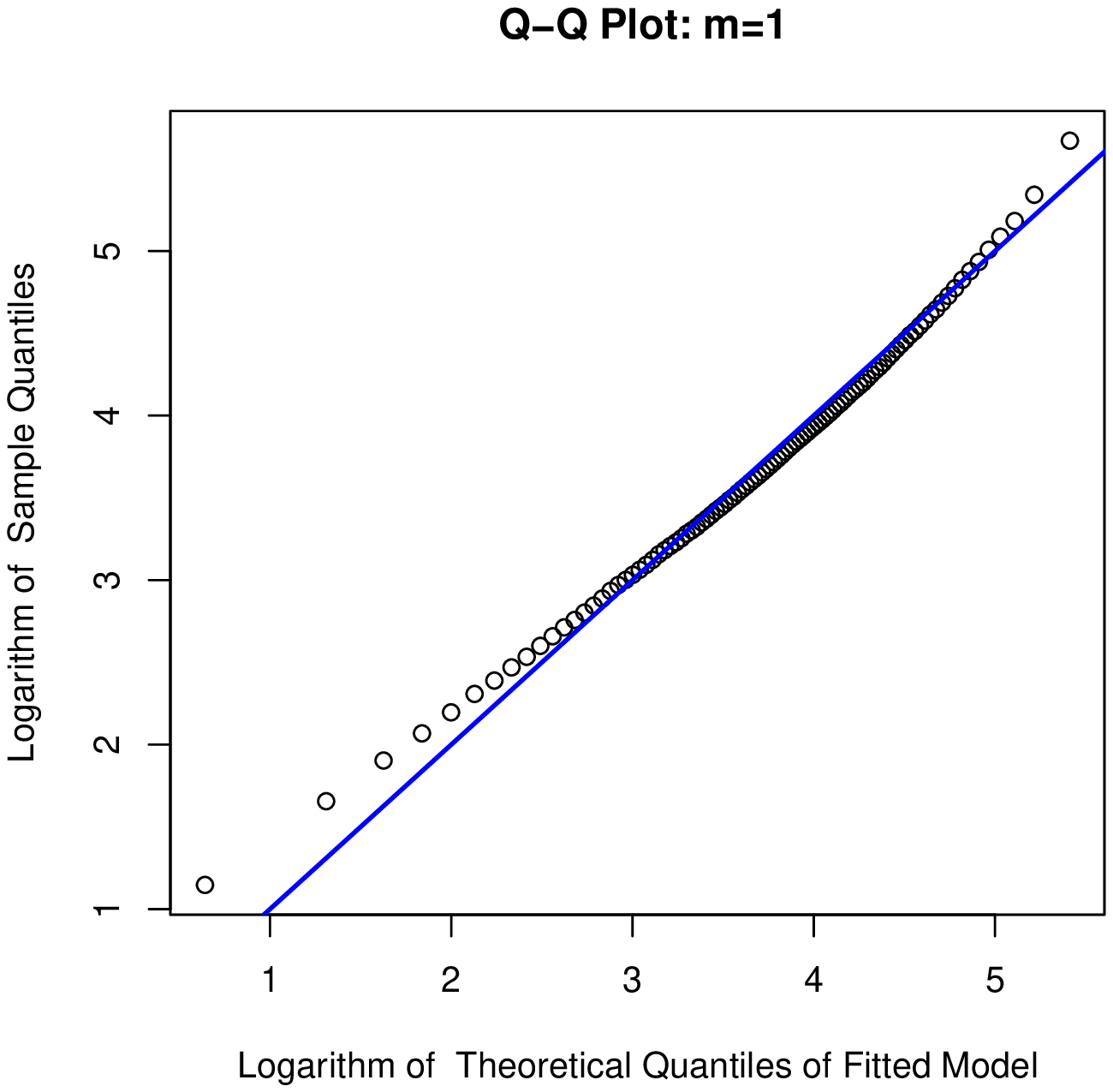}}
\hspace{0in}
{
\includegraphics[width=6cm]{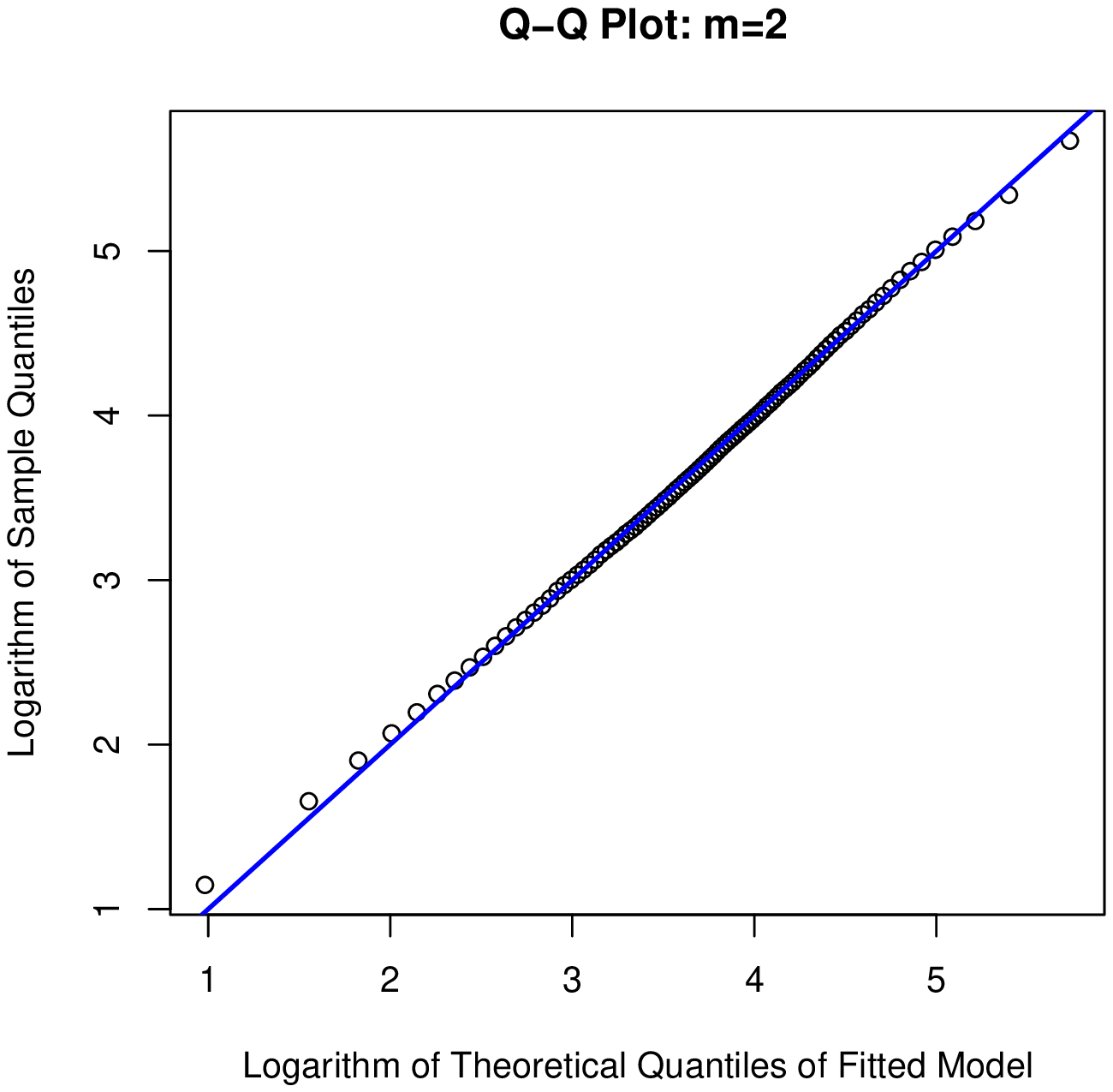}}
\hspace{0in}
{
\includegraphics[width=6cm]{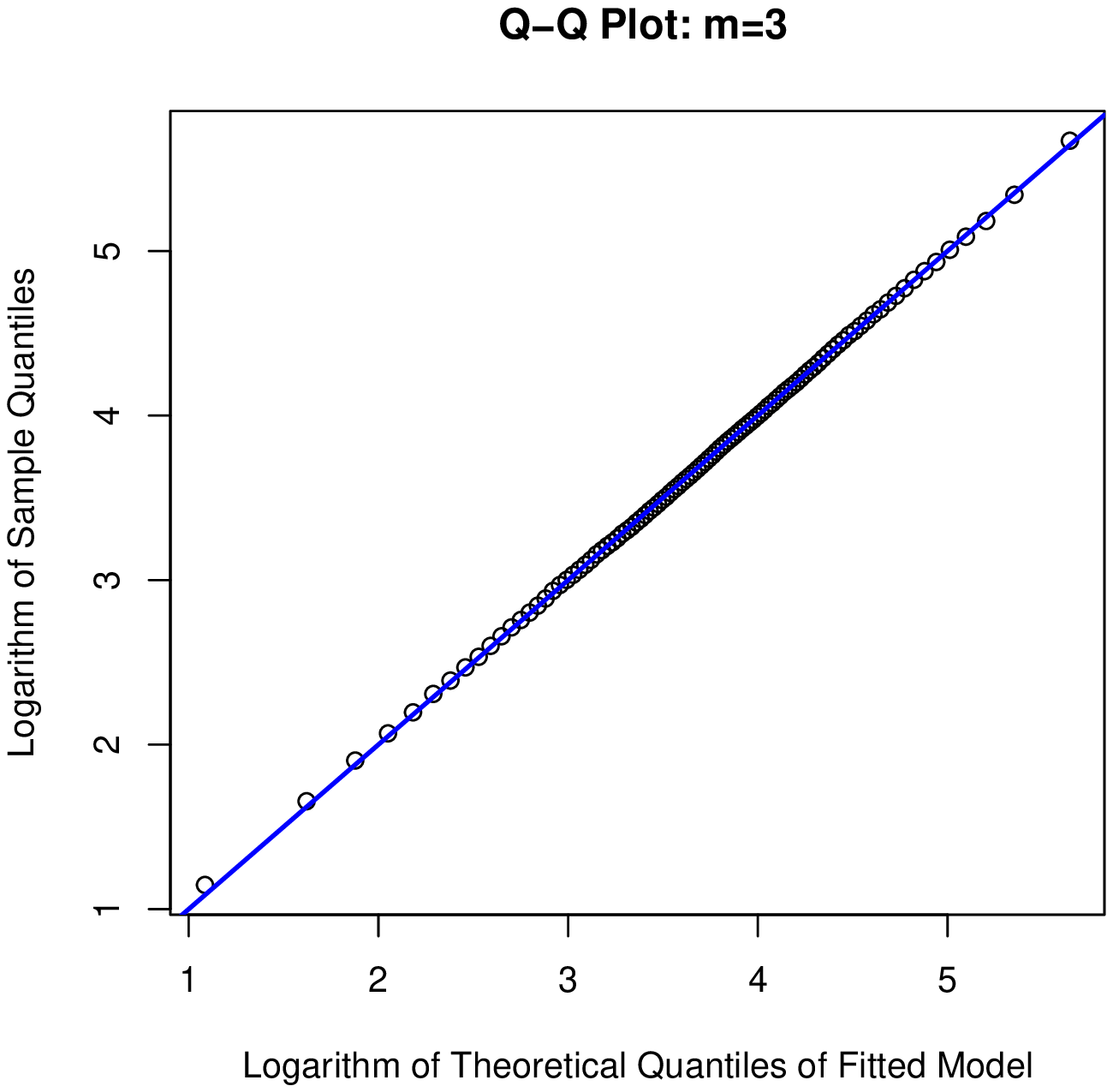}}
\hspace{0in}
{
\includegraphics[width=6cm]{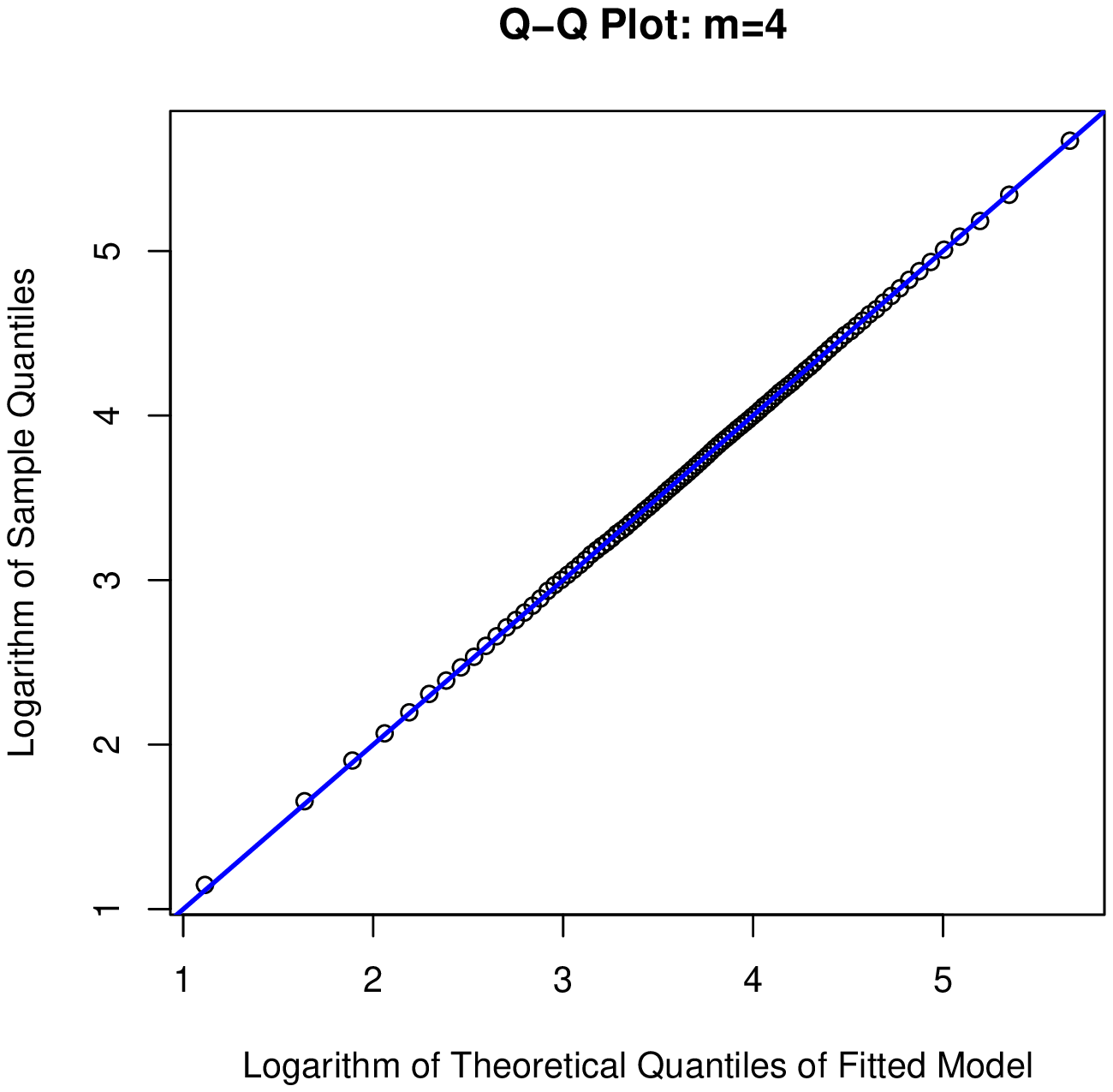}}

\hspace{0in}
\caption{QQ plots of fitted gamma mixtures.}
\label{Q-Q_plot_of_disposable_income_with_various_order_m}
\end{figure}

\section{Discussion and observations}
\label{sec7}
The finite Gamma mixture distributions are useful for modeling
positive data that is suspected to come from a heterogeneous population.
However, the MLE of the general Gamma mixture model is inconsistent.
We have shown that the MLE of the finite Gamma mixture model with
a structural shape parameter is strongly consistent.
The simulation results indicate that the MLE has respectable finite-sample properties
and observable consistency trends.
The real-data example demonstrates that this model
is able to reveal potential
subpopulation structure.

\section*{Acknowledgements}
We thank the China Institute for Income Distribution  for furnishing us with
the income data.
This work was supported by the
National Natural Science Foundation of China (Grant No.~11871419)
and the Natural Sciences and Engineering Research Council of Canada.

\bibliography{reference}

\begin{thebibliography}{28}
\expandafter\ifx\csname natexlab\endcsname\relax\def\natexlab#1{#1}\fi
\providecommand{\url}[1]{\texttt{#1}}
\providecommand{\href}[2]{#2}
\providecommand{\path}[1]{#1}
\providecommand{\DOIprefix}{doi:}
\providecommand{\ArXivprefix}{arXiv:}
\providecommand{\URLprefix}{URL: }
\providecommand{\Pubmedprefix}{pmid:}
\providecommand{\doi}[1]{\href{http://dx.doi.org/#1}{\path{#1}}}
\providecommand{\Pubmed}[1]{\href{pmid:#1}{\path{#1}}}
\providecommand{\bibinfo}[2]{#2}
\ifx\xfnm\relax \def\xfnm[#1]{\unskip,\space#1}\fi
\bibitem[{Chen and Chen(2003)}]{chen2003tests}
\bibinfo{author}{Chen, H.}, \bibinfo{author}{Chen, J.}, \bibinfo{year}{2003}.
\newblock \bibinfo{title}{Tests for homogeneity in normal mixtures in the
  presence of a structural parameter}.
\newblock \bibinfo{journal}{Statistica Sinica} \bibinfo{volume}{13},
  \bibinfo{pages}{351--365}.
\bibitem[{Chen(1998)}]{Chen1998}
\bibinfo{author}{Chen, J.}, \bibinfo{year}{1998}.
\newblock \bibinfo{title}{Penalized likelihood-ratio test for finite mixture
  models with multinomial observations}.
\newblock \bibinfo{journal}{The Canadian Journal of Statistics}
  \bibinfo{volume}{26}, \bibinfo{pages}{583--599}.
\bibitem[{Chen(2017)}]{chen2017}
\bibinfo{author}{Chen, J.}, \bibinfo{year}{2017}.
\newblock \bibinfo{title}{Consistency of the {MLE} under mixture models}.
\newblock \bibinfo{journal}{Statistical Science} \bibinfo{volume}{32},
  \bibinfo{pages}{47--63}.
\bibitem[{Chen et~al.(2016)Chen, Li and Tan}]{Chen2016}
\bibinfo{author}{Chen, J.}, \bibinfo{author}{Li, S.}, \bibinfo{author}{Tan,
  X.}, \bibinfo{year}{2016}.
\newblock \bibinfo{title}{Consistency of the penalized {MLE} for two-parameter
  gamma mixture models}.
\newblock \bibinfo{journal}{Science China Mathematics} \bibinfo{volume}{59},
  \bibinfo{pages}{2301--2318}.
\bibitem[{Chen and Tan(2009)}]{Chen2009}
\bibinfo{author}{Chen, J.}, \bibinfo{author}{Tan, X.}, \bibinfo{year}{2009}.
\newblock \bibinfo{title}{Inference for multivariate normal mixtures}.
\newblock \bibinfo{journal}{J. Multivar. Anal.} \bibinfo{volume}{100},
  \bibinfo{pages}{1367--1383}.
\newblock \DOIprefix\doi{10.1016/j.jmva.2008.12.005}.
\bibitem[{Chen et~al.(2008)Chen, Tan and Zhang}]{Chen2008}
\bibinfo{author}{Chen, J.}, \bibinfo{author}{Tan, X.}, \bibinfo{author}{Zhang,
  R.}, \bibinfo{year}{2008}.
\newblock \bibinfo{title}{Inference for normal mixtures in mean and variance}.
\newblock \bibinfo{journal}{Statistica Sinica} \bibinfo{volume}{18},
  \bibinfo{pages}{443--465}.
\bibitem[{CHIP13(2016)}]{CHIP13}
\bibinfo{author}{CHIP13}, \bibinfo{year}{2016}.
\newblock \bibinfo{title}{Chinese household income and expenditure project}.
\newblock
  \bibinfo{howpublished}{\url{http://www.ciidbnu.org/chip/chips.asp?year=2013}}.
\bibitem[{Ciuperca et~al.(2003)Ciuperca, Ridolfi and Idier}]{Ciuperca2003}
\bibinfo{author}{Ciuperca, G.}, \bibinfo{author}{Ridolfi, A.},
  \bibinfo{author}{Idier, J.}, \bibinfo{year}{2003}.
\newblock \bibinfo{title}{Penalized maximum likelihood estimator for normal
  mixtures}.
\newblock \bibinfo{journal}{Scandinavian Journal of Statistics}
  \bibinfo{volume}{30}, \bibinfo{pages}{45--59}.
\bibitem[{Dempster et~al.(1977)Dempster, Laird and Rubin}]{dempster1977maximum}
\bibinfo{author}{Dempster, A.P.}, \bibinfo{author}{Laird, N.M.},
  \bibinfo{author}{Rubin, D.B.}, \bibinfo{year}{1977}.
\newblock \bibinfo{title}{Maximum likelihood from incomplete data via the {EM}
  algorithm}.
\newblock \bibinfo{journal}{Journal of the Royal Statistical Society, Series B}
  \bibinfo{volume}{39}, \bibinfo{pages}{1--38}.
\bibitem[{Hathaway(1985)}]{Hathaway1985}
\bibinfo{author}{Hathaway, R.J.}, \bibinfo{year}{1985}.
\newblock \bibinfo{title}{A constrained formulation of maximum-likelihood
  estimation for normal mixture distributions}.
\newblock \bibinfo{journal}{The Annals of Statistics} \bibinfo{volume}{13},
  \bibinfo{pages}{795--800}.
\bibitem[{Kiefer and Wolfowitz(1956)}]{KW1956}
\bibinfo{author}{Kiefer, J.}, \bibinfo{author}{Wolfowitz, J.},
  \bibinfo{year}{1956}.
\newblock \bibinfo{title}{Consistency of the maximum likelihood estimator in
  the presence of infinitely many nuisance parameters}.
\newblock \bibinfo{journal}{The Annals of Mathematical Statistics}
  \bibinfo{volume}{27}, \bibinfo{pages}{887--906}.
\bibitem[{Li and Chen(2007)}]{Xinli2007}
\bibinfo{author}{Li, X.}, \bibinfo{author}{Chen, C.}, \bibinfo{year}{2007}.
\newblock \bibinfo{title}{Inequalities for the {G}amma function}.
\newblock \bibinfo{journal}{Journal of Inequalities in Pure and Applied
  Mathematics} \bibinfo{volume}{8}, \bibinfo{pages}{{article 28}}.
\bibitem[{Liu et~al.(2019)Liu, Li, Liu and Pu}]{LIU201929}
\bibinfo{author}{Liu, G.}, \bibinfo{author}{Li, P.}, \bibinfo{author}{Liu, Y.},
  \bibinfo{author}{Pu, X.}, \bibinfo{year}{2019}.
\newblock \bibinfo{title}{On consistency of the {MLE} under finite mixtures of
  location-scale distributions with a structural parameter}.
\newblock \bibinfo{journal}{Journal of Statistical Planning and Inference}
  \bibinfo{volume}{199}, \bibinfo{pages}{29--44}.
\bibitem[{Liu et~al.(2003)Liu, Pasarica and Shao}]{liu2003testing}
\bibinfo{author}{Liu, X.}, \bibinfo{author}{Pasarica, C.},
  \bibinfo{author}{Shao, Y.}, \bibinfo{year}{2003}.
\newblock \bibinfo{title}{Testing homogeneity in {G}amma mixture models}.
\newblock \bibinfo{journal}{Scandinavian Journal of Statistics}
  \bibinfo{volume}{30}, \bibinfo{pages}{227--239}.
\bibitem[{McLachlan et~al.(2019)McLachlan, Lee and
  Rathnayake}]{mclachlan2019finite}
\bibinfo{author}{McLachlan, G.J.}, \bibinfo{author}{Lee, S.X.},
  \bibinfo{author}{Rathnayake, S.I.}, \bibinfo{year}{2019}.
\newblock \bibinfo{title}{Finite mixture models}.
\newblock \bibinfo{journal}{Annual review of statistics and its application}
  \bibinfo{volume}{6}, \bibinfo{pages}{355--378}.
\bibitem[{Pearson(1894)}]{pearson1894}
\bibinfo{author}{Pearson, K.}, \bibinfo{year}{1894}.
\newblock \bibinfo{title}{Contributions to the mathematical theory of
  evolution}.
\newblock \bibinfo{journal}{Philosophical Transactions of the Royal Society of
  London. A} \bibinfo{volume}{185}, \bibinfo{pages}{71--110}.
\bibitem[{Pfanzagl(1988)}]{Pfanzagl1988}
\bibinfo{author}{Pfanzagl, J.}, \bibinfo{year}{1988}.
\newblock \bibinfo{title}{Consistency of maximum likelihood estimators for
  certain nonparametric families, in particular: mixtures}.
\newblock \bibinfo{journal}{Journal of Statistical Planning and Inference}
  \bibinfo{volume}{19}, \bibinfo{pages}{137--158}.
\bibitem[{Redner(1981)}]{Redner1981}
\bibinfo{author}{Redner, R.}, \bibinfo{year}{1981}.
\newblock \bibinfo{title}{Note on the consistency of the maximum likelihood
  estimate for nonidentifiable distributions}.
\newblock \bibinfo{journal}{The Annals of Statistics} \bibinfo{volume}{9},
  \bibinfo{pages}{225--228}.
\bibitem[{Tanaka(2009)}]{Tanaka2009}
\bibinfo{author}{Tanaka, K.}, \bibinfo{year}{2009}.
\newblock \bibinfo{title}{Strong consistency of the maximum likelihood
  estimator for finite mixtures of location-scale distributions when penalty is
  imposed on the ratios of the scale parameters}.
\newblock \bibinfo{journal}{Scandinavian Journal of Statistics}
  \bibinfo{volume}{36}, \bibinfo{pages}{171--184}.
\bibitem[{Tanaka and Takemura(2005)}]{Tanaka2005}
\bibinfo{author}{Tanaka, K.}, \bibinfo{author}{Takemura, A.},
  \bibinfo{year}{2005}.
\newblock \bibinfo{title}{Strong consistency of {MLE} for finite uniform
  mixtures when the scale parameters are exponentially small}.
\newblock \bibinfo{journal}{Annals of the Institute of Statistical Mathematics}
  \bibinfo{volume}{57}, \bibinfo{pages}{1--19}.
\bibitem[{Tanaka and Takemura(2006)}]{Tanaka2006}
\bibinfo{author}{Tanaka, K.}, \bibinfo{author}{Takemura, A.},
  \bibinfo{year}{2006}.
\newblock \bibinfo{title}{Strong consistency of the maximum likelihood
  estimator for finite mixtures of location-scale distributions when the scale
  parameters are exponentially small}.
\newblock \bibinfo{journal}{Bernoulli} \bibinfo{volume}{12},
  \bibinfo{pages}{1003--1017}.
\bibitem[{Teicher(1963)}]{Henry1963}
\bibinfo{author}{Teicher, H.}, \bibinfo{year}{1963}.
\newblock \bibinfo{title}{Identifiability of finite mixtures}.
\newblock \bibinfo{journal}{The Annals of Mathematical Statistics}
  \bibinfo{volume}{34}, \bibinfo{pages}{1265--1269}.
\bibitem[{van~der Vaart(2000)}]{vanderVaart2000}
\bibinfo{author}{van~der Vaart, A.W.}, \bibinfo{year}{2000}.
\newblock \bibinfo{title}{Asymptotic Statistics}.
\newblock \bibinfo{publisher}{Cambridge University Press, New York}.
\bibitem[{Wald(1949)}]{Wald1949}
\bibinfo{author}{Wald, A.}, \bibinfo{year}{1949}.
\newblock \bibinfo{title}{Note on the consistency of the maximum likelihood
  estimate}.
\newblock \bibinfo{journal}{The Annals of Mathematical Statistics}
  \bibinfo{volume}{20}, \bibinfo{pages}{595--601}.
\bibitem[{Willmot and Lin(2011)}]{willmot2011risk}
\bibinfo{author}{Willmot, G.E.}, \bibinfo{author}{Lin, X.S.},
  \bibinfo{year}{2011}.
\newblock \bibinfo{title}{Risk modelling with the mixed {E}rlang distribution}.
\newblock \bibinfo{journal}{Applied Stochastic Models in Business and Industry}
  \bibinfo{volume}{27}, \bibinfo{pages}{2--16}.
\bibitem[{Wong and Li(2014)}]{wong2014test}
\bibinfo{author}{Wong, S.}, \bibinfo{author}{Li, W.}, \bibinfo{year}{2014}.
\newblock \bibinfo{title}{Test for homogeneity in {G}amma mixture models using
  likelihood}.
\newblock \bibinfo{journal}{Computational Statistics and Data Analysis}
  \bibinfo{volume}{70}, \bibinfo{pages}{127--137}.
\bibitem[{Wu(1983)}]{Wu1983}
\bibinfo{author}{Wu, C.F.J.}, \bibinfo{year}{1983}.
\newblock \bibinfo{title}{On the convergence properties of the {EM} algorithm}.
\newblock \bibinfo{journal}{The Annals of Statistics} \bibinfo{volume}{11},
  \bibinfo{pages}{95--103}.
\bibitem[{Yin et~al.(2019)Yin, Lin, Huang and Yuan}]{yin2019consistency}
\bibinfo{author}{Yin, C.}, \bibinfo{author}{Lin, X.S.}, \bibinfo{author}{Huang,
  R.}, \bibinfo{author}{Yuan, H.}, \bibinfo{year}{2019}.
\newblock \bibinfo{title}{On the consistency of penalized {MLE}s for {E}rlang
  mixtures}.
\newblock \bibinfo{journal}{Statistics \& Probability Letters}
  \bibinfo{volume}{145}, \bibinfo{pages}{12--20}.

\end{thebibliography}


\begin{thebibliography}{23}
\expandafter\ifx\csname natexlab\endcsname\relax\def\natexlab#1{#1}\fi
\expandafter\ifx\csname url\endcsname\relax
  \def\url#1{\texttt{#1}}\fi
\expandafter\ifx\csname urlprefix\endcsname\relax\def\urlprefix{URL }\fi

\bibitem[{Chen and Chen(2003)}]{chen2003tests}
Chen, H., Chen, J., 2003. Tests for homogeneity in normal mixtures in the
  presence of a structural parameter. Statistica Sinica 13, 351--365.

\bibitem[{Chen(1995)}]{Chen1995Optimal}
Chen, J., 1995. Optimal rate of convergence for finite mixture models. Ann.
  Statist.

\bibitem[{Chen(1998)}]{Chen1998}
Chen, J., 1998. Penalized likelihood-ratio test for finite mixture models with
  multinomial observations. The Canadian Journal of Statistics / La Revue
  Canadienne de Statistique 26~(4), 583--599.

\bibitem[{Chen(2017)}]{Chen2017}
Chen, J., 2017. Consistency of the mle under mixture models. Statistical
  Science 32~(1), 47--63.

\bibitem[{Chen et~al.(2016)Chen, Li, and Tan}]{Chen2016}
Chen, J., Li, S., Tan, X., 2016. Consistency of the penalized mle for
  two-parameter gamma mixture models. Science China Mathematics 59~(12),
  2301--2318.

\bibitem[{Chen and Tan(2009)}]{Chen2009}
Chen, J., Tan, X., 2009. Inference for multivariate normal mixtures. J.
  Multivar. Anal. 100~(7), 1367--1383.

\bibitem[{Chen et~al.(2008)Chen, Tan, and Zhang}]{Chen2008}
Chen, J., Tan, X., Zhang, R., 2008. Inference for normal mixtures in mean and
  variance. Statistica Sinica 18~(2), 443--465.

\bibitem[{CHIP13(2016)}]{CHIP13}
CHIP13, 2016. Chip13 household income and expenditure.
  \url{http://www.ciidbnu.org/chip/chips.asp?year=2013}.

\bibitem[{Ciuperca et~al.(2003)Ciuperca, Ridolfi, and Idier}]{Ciuperca2003}
Ciuperca, G., Ridolfi, A., Idier, J., 2003. Penalized maximum likelihood
  estimator for normal mixtures. Scandinavian Journal of Statistics 30~(1),
  45--59.

\bibitem[{Dempster et~al.(1977)Dempster, Laird, and
  Rubin}]{dempster1977maximum}
Dempster, A.~P., Laird, N.~M., Rubin, D.~B., 1977. Maximum likelihood from
  incomplete data via the em algorithm. Journal of the royal statistical
  society. Series B 39, 1--38.

\bibitem[{Hathaway(1985)}]{Hathaway1985}
Hathaway, R.~J., 1985. A constrained formulation of maximum-likelihood
  estimation for normal mixture distributions. The Annals of Statistics 13~(2),
  795--800.

\bibitem[{Kiefer and Wolfowitz(1956)}]{KW1956}
Kiefer, J., Wolfowitz, J., 1956. Consistency of the maximum likelihood
  estimator in the presence of infinitely many nuisance parameters. The Annals
  of Mathematical Statistics 27~(4), 887--906.

\bibitem[{Li and Chen(2007)}]{Xinli2007}
Li, X., Chen, C., 2007. Inequalities for the gamma function. Journal of
  Inequalities in Pure and Applied Mathematics 8~(1), article28.

\bibitem[{Liu et~al.(2019)Liu, Li, Liu, and Pu}]{LIU201929}
Liu, G., Li, P., Liu, Y., Pu, X., 2019. On consistency of the mle under finite
  mixtures of location-scale distributions with a structural parameter. Journal
  of Statistical Planning and Inference 199, 29--44.

\bibitem[{Pearson(1894)}]{pearson1894}
Pearson, K., 1894. Contributions to the mathematical theory of evolution.
  Philosophical Transactions of the Royal Society of London. A 185, 71--110.

\bibitem[{Pfanzagl(1988)}]{Pfanzagl1988}
Pfanzagl, J., 1988. Consistency of maximum likelihood estimators for certain
  nonparametric families, in particular: mixtures. Journal of Statistical
  Planning and Inference 19~(2), 137--158.

\bibitem[{Redner(1981)}]{Redner1981}
Redner, R., 1981. Note on the consistency of the maximum likelihood estimate
  for nonidentifiable distributions. The Annals of Statistics 9~(1), 225--228.

\bibitem[{Tanaka(2009)}]{Tanaka2009}
Tanaka, K., 2009. Strong consistency of the maximum likelihood estimator for
  finite mixtures of location-scale distributions when penalty is imposed on
  the ratios of the scale parameters. Scandinavian Journal of Statistics
  36~(1), 171--184.

\bibitem[{Tanaka and Takemura(2005)}]{Tanaka2005}
Tanaka, K., Takemura, A., 2005. Strong consistency of mle for finite uniform
  mixtures when the scale parameters are exponentially small. Annals of the
  Institute of Statistical Mathematics 57~(1), 1--19.

\bibitem[{Tanaka and Takemura(2006)}]{Tanaka2006}
Tanaka, K., Takemura, A., 2006. Strong consistency of the maximum likelihood
  estimator for finite mixtures of location-scale distributions when the scale
  parameters are exponentially small. Bernoulli 12~(6), 1003--1017.

\bibitem[{Teicher(1963)}]{Henry1963}
Teicher, H., 1963. Identifiability of finite mixtures. The Annals of
  Mathematical Statistics 34~(4), 1265--1269.

\bibitem[{Wald(1949)}]{Wald1949}
Wald, A., 1949. Note on the consistency of the maximum likelihood estimate. The
  Annals of Mathematical Statistics 20~(4), 595--601.

\bibitem[{Wu(1983)}]{Wu1983}
Wu, C. F.~J., 1983. On the convergence properties of the em algorithm. The
  Annals of Statistics 11~(1), 95--103.

\end{thebibliography}

\end{document}